\documentclass[10pt]{article}

\usepackage[colorlinks,linkcolor=blue,citecolor=blue]{hyperref}
\usepackage{amsthm}
\usepackage{multirow}
\usepackage{marginnote}
\usepackage{a4wide}
\usepackage{amssymb}
\usepackage{amsfonts}
\usepackage{amsmath}
\usepackage{mathrsfs}
\usepackage{tikz}
\usetikzlibrary{arrows,matrix}
\usetikzlibrary{positioning}
\usepackage{mdframed} 
\usepackage{lipsum} 
\usepackage{extarrows} 
\usepackage{color}
\usepackage{enumitem} 
\usepackage{tocloft} 
\usepackage{titlesec} 

\usepackage[all]{xy} 

\input xy
\xyoption{arrow} \xyoption{matrix}

\date{}

\newtheorem{proposition}{Proposition}[section]
\newtheorem{theorem}[proposition]{Theorem}
\newtheorem{lemma}[proposition]{Lemma}

\newtheorem{definition}[proposition]{Definition}
\newtheorem{corollary}[proposition]{Corollary}

\def\der{\partial }

\def\nFM0{{\nu }_{F,M_0}}
\def\nFN0{{\nu }_{F,N_0}}
\def\nGN0{{\nu }_{G,N_0}}

\def\N0{ {\bf N}_0 }

\def\t{\otimes}
\def\g{\gamma}

\def\ra{\rightarrow}

\def\Xpm{X^{\pm }}

\def\s{\sigma}

\def\l1{{\lambda}_1}

\def\a{\alpha}
\def\a0{ {\alpha }_0}
\def\a1{ {\alpha }_1}

\def\l{\lambda}
\def\o{\omega}

\def\nFGM0{{\nu }_{F,G,M_0}}

\def\nFN0{{\nu}_{F,N_0}}

\def\sm{{\sigma}^m}

\def\sm1{{\sigma}^{-1}}

\def\smtp1{{\sigma}^{-t+1}}

\def\o{\omega }
\def\S1{S^{-1}}

\def\Xpm1{X^{\pm 1}_1}

\def\sPM1{{\sigma }^{\pm 1}}
\def\sMP1{{\sigma }^{\mp 1 }}

\def\b{\beta}

\def\di{{\rm d.ind}}

\def\L{\Lambda}

\def\G{\Gamma}

\def\CA{{\cal A}}

\def\CD{{\cal D}}

\def\Ytm1{Y^{t-1}}
\def\Yim1{Y^{i-1}}

\def\CN{{\cal N}}

\def\CF{{\cal F}}
\def\CG{{\cal G}}

\def\supp{{\rm supp }}

\def\Aut{{\rm Aut}}

\def\dim{{\rm dim }}

\def\gr{ {\rm gr} }

\def\SL2Z{ {\rm SL}_2({\bf Z}) }

\def\Gp1{ G^{1 , 1 } }
\def\P11{ P^{-1 , 1 } }
\def\Pp1{ P^{1 , 1 } }

\def\nCLsr{{}^\nu\kern-2pt {\cal L}^{\sigma , \rho  }}
\def\nP{{}^\nu \kern-2pt P}
\def\nL{{}^\nu\kern-2pt L}
\def\nLL{{}^\nu\kern-2pt \Lambda}
\def\nPsr{{}^\nu\kern-2pt P^{\sigma , \rho  }}
\def\nLsr{{}^\nu\kern-2pt L^{\sigma , \rho  }}
\def\nuCL{{}^\nu\kern-2pt  {\cal L}}
\def\nCLsr{{}^\nu\kern-2pt {\cal L}^{\sigma , \rho  }}
\def\nCL1m{{}^\nu\kern-2pt {\cal L}^{-1 , 1  }}
\def\x1nu{x^\frac{1}{\nu}}
\def\xm1nu{x^{-\frac{1}{\nu}}}

\def\CN{{\cal N}}
\def\ra{\rightarrow }

\def\CB{{\cal B}}

\def\nAM0{{\nu }_{{\cal A},M_0}}
\def\nAN0{{\nu }_{{\cal A},N_0}}

\def\End{ {\rm End }}

\def\gr{\mathfrak{r}}

\def\SL{{\rm SL}}

\def\di!{\frac{\der^i}{i!}}
\def\dik!{\frac{\der^k_i}{k!}}

\def\gl{\mathfrak{l}}

\def\Max{{\rm Max}}

\def\N{\mathbb{N}}

\def\0{\overline{0}}
\def\1{\overline{1}}

\def\Ln1{\L_{n,\overline{1}}}

\def\a1{a_{\overline{1}}}

\def\S{\Sigma}

\def\vn1{\overrightarrow{n-1}}

\def\im{{\rm im}}

\def\gl{{\rm gl}}
\def\sl{{\rm sl}}

\def\mJ{\mathbb{J}}
\def\mI{\mathbb{I}}

\def\ann{{\rm ann}}

\def\K1{{\rm K}_1}

\def\hmI1{\widehat{\mI_1}}
\def\tmI1{\widetilde{\mI_1}}
\def\tmJ1{\widetilde{\mJ_1}}
\def\hB1{\widehat{B_1}}
\def\hCB1{\widehat{\CB_1}}

\def \S{\mathcal{S}}

\def\sl2{\mathfrak{sl}_2}

\def\Deg{{\rm Deg}}
\def\Ind{{\rm Ind}}

\def\sl2{\mathfrak{sl}_2}
\def\gl2{\mathfrak{gl}_2}

\def\b1{\overline{1}}

\def\fC{{\mathfrak{C}}}
\def\fCK{{\mathfrak{C}}(K)}

\def\fCdK{{\mathfrak{C}}_d(K)}

\def\gl{{\mathfrak{l}}}

\makeatletter
\newenvironment{proof*}[1][\proofname]{\par
  \pushQED{\qed}%
  \normalfont \partopsep=\z@skip \topsep=\z@skip
  \trivlist
  \item[\hskip\labelsep
        \itshape
    #1\@addpunct{.}]\ignorespaces
}{%
  \popQED\endtrivlist\@endpefalse
}
\makeatother

\begin{document}

\author{V. V. \  Bavula 
}
\title{ The Galois Theory (a ring theoretic approach)
}

\maketitle

\begin{abstract}
The fundamental concepts in the Galois Theory are  
 separable, normal   and  Galois field extensions. These concepts are central in proofs of the Galois Theory.  
In the paper, we introduce a new approach, a ring theoretic approach, to the Galois Theory which is based on central simple algebras and none of the above concepts are used or even mentioned. The only concept which is used is `G-extension' (a finite field extension $L/K$ is called a {\em G-extension}  if the endomorphism algebra $\End_K(L)$ is generated by the field $L$ and the automorphism group $\Aut_{K-{\rm alg}}(L)$). So, G-extensions are the most symmetric field extensions. In this approach, the Galois Correspondences (for subfields and Galois subfields of $L$) are deduced from the Double Centralizer Theorem which is applied to the central simple algebra  $\End_K(L)$ and G-extensions. Since the class of G-extensions {\em coincides} with the class of Galois extensions, all main results of the Galois Theory are obtained from the `corresponding' results for G-extensions. This approach gives a new conceptual (short) proofs of key results of the Galois Theory. It also reveals that the `maximal symmetry' (of field extensions) is the essence of the Galois Theory. 
 
$\noindent$


$\noindent$


{ \small \tableofcontents}

\end{abstract}


\section{Introduction} \label{INTR} 

In this paper, module means a {\em left} module.  The following
notation will remain fixed throughout the paper (if it is not
stated otherwise): 

\begin{itemize}

\item $L/K$ is a finite field extension,

\item  $\CF (L/K)$ and  $\CG (L/K)$ are the sets of all and Galois subfields of $L/K$, respectively,

\item   $G:=G(L/K)$ is the automorphism group of the $K$-algebra $L$, $\CG (G(L/K))$ and  $\CN(G(L/K))$ are  the sets of all and normal subgroups of $G(L/K)$, respectively.  If $L/K$ is a Galois field extension then  $G(L/K)$ is called the Galois group of $L/K$, 

\item $L^{G(L/K)}:=\{ l\in L\, | \, \s (l)=l$ for all automorphisms $\s\in G(L/K)\}$ is the field of $G(L/K)$-invariants in $L$,

\item For a finite field extension $L/K$ and its intermediate subfield $K\subseteq M\subseteq L$, let
$$
G(L/K)^M:= \{ g\in G(L/K)\, | \, g(m)=m\;\; {\rm for\; all}\;\; m\in M \}=C_{G(L/K)}(M),
$$

\item $E:=E(L/K):=\End (L/K):=\End_K (L)\simeq M_n(K)$ is the endomorphism algebra of the $K$-vector space $L$ and  $M_n(K)$ is the algebra of $n\times n$ matrices over $K$ where $n=[L:K]:=\dim_K(L)$ is the degree of the field extension $L/K$,

\item $\CA (E,L)$ is the set of all $K$-subalgebras of $E$ that contain the field $L$, 

\item For  an algebra $A\in  \CA (E,L)$,  let
$$ 
L^{A\cap G(L/K)}:= \{  l \in L\, | \, g(l)=l\;\; {\rm for\; all}\;\; g\in A\cap G(L/K) \},
$$
\item $\CA (E,L, G):=\{ A\in  \CA (E,L)\, | \, gAg^{-1}\subseteq A$ for all $g\in G\}$, 

\item $\fCK$ is  the class of   central simple finite dimensional  $K$-algebras and  $\fCdK$ is  the class of  central simple finite dimensional  division $K$-algebras,


\end{itemize}

Let $K$ be  field, $L/K$ be a finite field extension, $G=G(L/K)$ be the automorphism group of the $K$-algebra $L$ and $E=E(L/K)$ be the algebra of $K$-endomorphisms of the $K$-module $L$. The $K$-algebra $E\simeq M_{[L:K]}(K)$ is a central simple $K$-algebra where $[L:K]:=\dim_K(L)$ is the {\em degree} of $L$ over $K$. The algebra $E$ is the `universe' in which all the major objects of the  paper leave and interact and the reason why we use some  results for central simple algebras. In particular, $G\subseteq E$ and the field $L$ is a $K$-subalgebra of $E$ via the $K$-algebra  monomorphism: 
$$
L\ra E, \;\; l\mapsto l\cdot: L\ra L, \;\; \l \mapsto l\l.
$$
A subalgebra of $E$ which is generated by the field $L$ and the group $G$ is the {\em skew  group algebra} (Theorem \ref{BC24Mar25}.(2)) 
$$
L\rtimes G=\bigoplus_{g\in G}Lg
$$
where the multiplication is given by the rule: For all elements $\l,\mu \in L$ and $g,h\in G$,
$$
 \l g \cdot \mu h=\l g(\mu)gh.
 $$

\begin{definition}
 A finite field extension $L/K$ is called a {\bf G-extension} if 
 $$E=L\rtimes G.$$
 \end{definition}
By the definition, G-extensions are the {\em most symmetric} field extensions (as far as their automorphism groups are concerned). This also follows  from the fact that:  
  {\em A finite field extension  is a G-extension iff it is a Galois field extension} (Proposition \ref{A8May25}).
  
  The following two results are the main results of the classical Galois Theory: {\em Let a finite field extension $L/K$ be Galois. Then:}
  \begin{enumerate}

\item {\bf (The Galois Correspondence for subfields)} 
{\em The map
$$
\CF (L/K)\ra \CG (G(L/K)), \;\; M\mapsto  G(L/M)
$$ 
is a bijection with inverse} $
H\mapsto L^H.$

\item {\bf (The  Galois Correspondence for Galois subfields)} 
{\em The map
$$
\CG (L/K)\ra \CN (G(L/K)), \;\; \G\mapsto  G(L/\G)
$$ 
is a bijection with inverse} $
\G\mapsto L^{\G}.$

\end{enumerate}

We prove these results using {\em only} the definition of G-extension without mentioning the key ingredients of the Galois Theory 
 (separable and normal  field extensions and the Galois field extension). \\

{\bf Key ideas and steps behind the proofs}: For  subset $S\subseteq E$, $C_E(S):=\{ a\in E\, | \, as=sa$ for all $s\in S\}$ is  the {\em centralizer} of $S$ in $E$. The centralizer $C_E(S)$ is a $K$-algebra.

\begin{itemize}

\item (Lemma \ref{ac1Jun25}) {\em For every field $M\in \CF (L/K)$, the field extension $L/M$ is a G-extension.}

Let $\CA (E,L)$ and $\CA (E,L, sim)$ be  the sets of all  and simple $K$-subalgebras of $E$ that contain the field $L$, respectively. The next result gives an explicit description of these sets and an explicit bijection between the sets $\CF (L/K)$ and  $\CA (E,L)$.

\item (Theorem \ref{B24Mar25}.(2,3)) $\CA ( E, L)=\CA ( E, L, sim)=\{ E(L/M)\, | \, M\in \CF (L/K) \}$ {\em and the map 
$$\CF (L/K)\ra \CA ( E, L), \;\; M\mapsto C_E(M)=E(L/M)$$  is a bijection with inverse} $B\mapsto C_E(B)$. 

Theorem \ref{B24Mar25}.(2,3) is the key result from which the two Galois correspondences follow by finding explicitly formulae for the centralizers $C_E(M)$ and $C_E(B)$ using the equalities 
$E=L\rtimes G$ and $E(L/M)=L\rtimes G(L/M)$, see Theorem \ref{1Jun25} below.

\item (Theorem \ref{1Jun25}) {\em Let $L/K$ be a  G-extension (i.e. a Galois extension). Then:}

\begin{enumerate}

\item $ \CA (E,L)=\{ C_{E}(M)=L\rtimes G(L/M)=L\rtimes G(L/K)^M\, | \, L\in \CF (L/K)\}$. 

\item {\em The map
$$
\CF (L/K)\ra \CA (E,L), \;\; M\mapsto  C_{E}(M)=L\rtimes G(L/M)=L\rtimes G(L/K)^M
$$ 
is a bijection with inverse}
$$
A\mapsto C_{E}(A)=L^{A\cap G(L/K)}.
$$
\item {\bf (The  Galois Correspondence)} 
{\em  The map
$$
\CF (L/K)\ra \CG (G(L/K)), \;\; M\mapsto  G(L/M)=G(L/K)^M
$$ 
is a bijection with inverse} $
H\mapsto L^H.$
\end{enumerate}
The  Galois Correspondence  is obtained from statement 2 by replacing the set $\CA (E,L) $ by the set  $\CG (G(L/K))$ and the algebra $L\rtimes  G(L/M)$ by the group $G(L/M)$ (by which the algebra 
$L\rtimes  G(L/M)$ is {\em uniquely} determined). 

Recall that   $\CG  (L/K)$ is the set of all Galois subfields of $L/K$ and 
$$
\CA (E, L, G):=\{ A\in \CA (E, L)\, | \, gAg^{-1}\subseteq A\;\; {\rm  for\; all}\;\;  g\in G\},
$$  
the set of $G$-{\em stable} subalgebras of $E$ that contain the field $L$. Theorem \ref{11Jun25}.(1) describes the set $\CA (E, L, G)$ in terms of the Galois subfields of $L$ and normal subgroups of $G$. Theorem \ref{11Jun25}.(2) gives a bijection between the sets $\CG (L/K)$ and $\CA (E, L, G)$.

\item (Theorem \ref{11Jun25})  {\em Let $L/K$ be a  G-extension (i.e. a Galois extension). Then:}

\begin{enumerate}

\item $\CA (E,L, G)=\{C_E(\G )=L\rtimes G(L/\G)\, | \,$ $ \G\in \CG  (L/K)  \}=\{L\rtimes N\, | \, N\in \CN  (G)  \}$  and $\CN (G)= \{ G(L/\G)\, | \,$ $ \G\in \CG  (L/K)  \}$.

\item {\em The map
$$
\CG (L/K)\ra  \CA (E,L, G), \;\; \G \mapsto  C_{E}(M)=L\rtimes G(L/\G) 
$$
 is a bijection with inverse}
$$ A\mapsto C_{E}(A)=L^{A\cap G}.
$$
\item {\bf (The  Galois Correspondence for Galois subfields)} 
{\em The map
$$
\CG (L/K)\ra \CN (G), \;\; \G\mapsto  G(L/\G)=G(L/K)^{\G}
$$ 
is a bijection with inverse} $
\G\mapsto L^{\G}.$

\end{enumerate}

The  Galois Correspondence for Galois subfields of $L$  is obtained from statement 2 by replacing the set $\CA (E,L, G) $ by the set  $\CN (G)$ and the algebra $L\rtimes  G(L/\G)$ by the group $G(L/\G)$ (by which the algebra 
$L\rtimes  G(L/\G)$ is {\em uniquely} determined). 

\end{itemize}

The paper is organized as follows. In Section \ref{AAEL}, we recall  some classical results on central simple algebras (the Noether-Skolem Theorem, the Double centralizer Theorem, etc). For a central simple algebra $A\in \fCK$ and its strongly maximal subfield $L$, Theorem 
 \ref{A24Mar25} is a description  of  subalgebras of $A$  that contain the field $L$. Applying it to the central simple algebra $E(L/K)$ and its strongly maximal subfield $L$, a description  of  subalgebras of $E(L/K)$  that contain the field $L$ is obtained (Theorem  \ref{B24Mar25}). We describe
   some of  the  properties of the algebra $L\rtimes G$ (Theorem \ref{BC24Mar25}) and the sets  $\CA (L\rtimes G, L)$ and $\CA (L\rtimes G, L, G)$ (Corollary \ref{a24Mar25}).

In Section \ref{AN-GT-NORMAL}, proofs are given to 
  Theorem \ref{1Jun25} and  Theorem \ref{11Jun25}.  
Several useful results on G-extensions are proven that are used in the proofs of the two theorems. \\

{\bf Generalizations:  B-extensions and analogue of the  Galois Theory for normal fields.}  In \cite{Jacobson-1937, Jacobson-1944}, Jacobson  introduced an analogue  of the Galois theory for {\em purely inseparable} field  extensions  $L/K$ of exponent 1 (meaning that the $p$'th power of every element of $L$ is in $K$), where the Galois groups (which are trivial) are replaced by restricted Lie algebras of derivations.

 A purely inseparable extension is called a {\bf modular extension} if it is a tensor product of simple extensions. In particular, every extension of exponent 1 is modular, but there are non-modular extensions of exponent 2 as shown by Weisfeld \cite{Weisfeld-1965}. Sweedler \cite{Sweedler-1968} and Gerstenhaber  and  Zaromp \cite{Gerstenhaber-Zaromp-1970} gave an extension of the Galois correspondence to {\em modular} purely inseparable extensions, where derivations are replaced by higher derivations.

For a finite field extension $L/K$, the algebra $E(L/K)$ contains the group $G(L/K)$ and the algebra $\CD (L/K)$ of differential operators on the $K$-algebra $L$. A $K$-subalgebra of $E(L/K)$ which generated by the group $G(L/K)$ and the algebra $\CD (L/K)$ is the skew group algebra 
$$
\CD (L/K)\rtimes G(L/K)=\bigoplus_{g\in G(L/K)}\CD (L/K)g\subseteq E(L/K),
$$
 \cite{CDA-NormalMaxFld}. {\em The field extension is separable iff} $\CD (L/K)=L$, \cite{CDA-NormalMaxFld}.

 \begin{definition}
(\cite{CDA-NormalMaxFld})  A finite field extension $L/K$ is called a {\bf B-extension} if 
 $$E(L/K)=\CD (L/K)\rtimes G(L/K).$$
 \end{definition}
 `B' stands for `Bi'=`2', i.e. both fundamental objects that are attached to the field extension $L/K$, the algebra of differential operators $\CD (L/K)$ and the automorphism group $G(L/K)$, play a key role in the definition. By the definition, 
 the  field  extension $L/K$ is a B-extension iff the algebra $\CD (L/K)\rtimes G(L/K)$ is the largest  possible. So, B-extensions are the most symmetrical field extensions.  Galois field extensions $L/K$ are precisely B-extensions with $E(L/K)=L\rtimes G(L/K)$. Purely inseparable field extensions $L/K$ are  precisely B-extensions with $E(L/K)=\CD (L/K)$, \cite{CDA-NormalMaxFld}. 
 Surprisingly, {\em the class of B-extensions coincide with the class of normal extensions,} \cite{CDA-NormalMaxFld}. This gives a new characterizeion of normal fields. In \cite{CDA-NormalMaxFld}, using a similar ideas an analogue of the Galois Theory is developed for B-extension, i.e. for normal field extensions.


\section{ The subalgebra $L\rtimes G(L/K)$ of the endomorphism algebra $E(L/K)$} \label{AAEL} 

 At the beginning of the section, we collect results on the central simple algebras that are used in the paper. The key results of the section are  Theorem 
 \ref{A24Mar25} and  Theorem  \ref{B24Mar25}.
 For a central simple algebra $A\in \fCK$ and its strongly maximal subfield $L\in \Max_s(A)$, Theorem 
 \ref{A24Mar25}  describes subalgebras of $A$  that contain the field $L$. For a finite field extension $L/K$,  Theorem \ref{B24Mar25} describes subalgebras of the endomorphism algebra $E$ that contain the field $L$. Theorem \ref{BC24Mar25} describes some of  the  properties of the algebra $L\rtimes G$. Corollary \ref{a24Mar25} describes the sets $\CA (L\rtimes G, L)$ and $\CA (L\rtimes G, L, G)$.\\

 {\bf Central simple algebras.}    Let us recall some basic facts and definitions on central simple algebras, see the book \cite{Pierce-AssAlg} for details. For a field $K$, we denote by $\fCK$  the class of   central simple finite dimensional  $K$-algebras. Let  $\fCdK$ the class of   central simple finite dimensional  division $K$-algebras. Clearly,  $\fCdK\subseteq \fCK $. A 
$K$-algebra belongs to the class $\fCK$ iff $A\simeq M_n(D)$ for some $D\in \fCdK$.

 The dimension of a division $K$-algebra $D$ over a field $K$ is denoted by $\dim_K(D)=[D:K]$.  For each algebra $\CA\in \fCK$, 
$$
\dim_K(\CA )=n^2\;\; {\rm for\; some\; natural\; number}\;\; n\geq 1
$$
 which is called the {\bf degree} of $\CA $ and is denoted by $\Deg (\CA )$. The algebra $\CA$ is isomorphic to the matrix algebra $M_l(D)$ over a division algebra $D\in \fCK$ which is unique (up to a $K$-isomorphism). Clearly, $\Deg (\CA ) = l\Deg (D)$. The degree $\Deg (D)$ is called the {\bf index} of $\CA$ and is denoted by $\Ind (\CA )$. If $L$ is a  subfield of $\CA$ that contains the field $K$ then 
 $$
 [L:K]\leq \Deg (\CA ).
 $$
  The field $L$ is called a {\bf strictly maximal} subfield if $[L:K]= \Deg (\CA )$. The sets of maximal and strictly maximal subfields of $A\in \fCK$ are denoted by $\Max (A)$ and $\Max_s(A)$, respectively. In general, the containment $\Max_s(A)\subseteq \Max (A)$ is strict but for all central division algebras $D\in \fCdK$,  the equality 
  $$
  \Max_s(D)= \Max (D)
  $$ 
  holds.

\begin{theorem}\label{Noether-SkolemThm}
{\bf (The Noether-Skolem Theorem)} Let $\CA\in \fCK$,  $A$ and $B$ be isomorphic,   simple subalgebras of $\CA$. Then $B=uAu^{-1}$ for some unit $u\in \CA^\times$.
\end{theorem} 
 
 For a subset $S$ of $A$, the set  
$C_A(S)=\{ a\in A\, | \, as=sa\}$ 
 is called the {\bf centralizer} of $S$ in $A$. The centralizer $C_A(S)$ is a subalgebra of $A$. 
 The algebra $C_AC_A(S)=C_A(C_A(S))$ is called the {\em double centralizer} of $S$.
 Notice that $S\subseteq C_AC_A(S)$. Clearly, 
$C_AC_AC_A(S)=C_A(S)$.

The dimension of a division $K$-algebra $D$ over a field $K$ is denoted by $\dim_K(D)=[D:K]$.

\begin{theorem}\label{DCThm}
Let $A\in \fCK$ and $B$ be a simple subalgebra of $A$.
\begin{enumerate}
\item {\bf (Double Centralizer Theorem)} $C_AC_A(B)=B$.
\item The algebra $C_A(B)$ is a simple algebra. 
\item $\dim_K(B)\, \dim_K(C_A(B))=\dim_K(A)$.
\item Suppose that $B\in \fCK$. Then $C_A(B)\in \fCK$ and $A=B\t C_A(B)$.
\end{enumerate}
\end{theorem}
 
\begin{definition}
Let $A\in \fCK$ and $B$ be a simple subalgebra of $A$. The the pair $(B,C_A(B))$ (of necessarily  simple)  $K$-subagebras of $A$ is called {\bf a double centralizer pair}, a {\bf dcp},  for short.
\end{definition}

If $(B,C_A(B))$ is a dcp then so is $(C_A(B),B)$, and vice versa.\\

 {\bf Description of subalgebras of the endomorphism algebra $E(L/K)$ that contain the field $L$.} 
 Let $A$ be an algebra. For an $A$-module  $M$, the set $\ann_A(M):=\{ a\in A\, | \, aM=0\}$ is called the {\em annihilator} of $M$. The annihilator $\ann_A(M)$ is an ideal of the algebra $A$. A module  is called a {\em faithful} module if its annihilator is equal to zero.

\begin{lemma}\label{a5May25}
A finite dimensional algebra $A$ is a simple algebra iff there is a faithful simple $A$-module. 
\end{lemma}

\begin{proof} $(\Rightarrow)$ All nonzero modules of a simple algebra are faithful.

$(\Leftarrow)$ Let $\gr$ be the radical of the finite dimensional  algebra $A$.  
 Suppose that $M$ is a faithful simple $A$-module. Then  $\gr M=0$ (by the simplicity of $M$ and  definition of the radical). This  implies that $\gr =0$, by the faithfulness of the $A$-module $M$. So, the algebra $A$ is a semisimple finite dimensional algebra, i.e. it is  a finite  direct product of simple finite dimensional algebras. For a semisimple  but not simple finite dimensional algebra, every simple module is unfaithful. Therefore, the algebra $A$ is a simple algebra. 
\end{proof}

 For an algebra $A$ and its subalgebra $B$, let $\CA (A, B)$ (resp., $\CA (A, B, sim)$) be the sets of all (resp., simple) subalgebras of $A$ that contain $B$. 

\begin{proposition}\label{A24Mar25}
(\cite{CDA-NormalMaxFld}) Let $K$ be a field,   $A\in \fCK$ and $L$ is a strongly maximal subfield of $A$. Then:

\begin{enumerate}

\item   $\CA ( E, L)=\CA ( E, L, sim)$.

\item The map 
$$\CF (L/K)\ra \CA ( E, L)=\CA (A, L, sim), \;\; M\mapsto C_A(M)$$  is a bijection with inverse $B\mapsto C_A(B)$.

\end{enumerate}
\end{proposition}

\begin{proof} 1. We have to show that the algebra $B$ is a simple algebra. The inclusion $L\subseteq B$ implies the inclusion 
$$
C_E(B)\subseteq C_E (L)=L.
$$
 The last equality follows from the fact that $L\in \Max_s (E)$ (statement 1). Clearly, the field $L$ is a faithful simple $E$-module. Since $L\subseteq B$, the field $L$ is also a faithful simple $B$-module. By Lemma \ref{a5May25}, the finite dimensional algebra $B$ is a simple algebra.

2. Let us show that the map is a well-defined map. Since the field $M/K\in \CF (L/K)$ is a simple algebra, its centralizer $C_A(M)$ is also a simple subalgebra of $A$, by Theorem \ref{DCThm}.(2). Clearly, $L\subseteq C_A(M)$  (since $M\subseteq L$), and so $C_A(M)\in \CA:= \CA (A, L, sim)$. 

The field $L$ is a strongly maximal subfield of $A$. Hence, $C_A(L)=L$. For every algebra $B\in \CA$, $ L\subseteq B$. Hence, $C_A(B)\subseteq C_A(L)=L$, and so 
$$
C_A(B)\in \CF (L/K).
$$ 
By the Double Centralizer Theorem (Theorem \ref{DCThm}.(1)), $C_A(C_A(M))=M$   for all $M\in \CF (L/K)$ and $C_A(C_A(B))=B$ for all $B\in \CA$, and the result follows. 
\end{proof}

Theorem \ref{B24Mar25} is a  description of subalgebras of the endomorphism algebra $E$ that contain the field $L$. This is one of the key results of the paper.

\begin{theorem}\label{B24Mar25}
(\cite{CDA-NormalMaxFld})  Let 
 $L/K$ be a finite field extension of degree $n:=[L:K]<\infty$ and  $E:=E(L/K )\simeq M_n(K)$. Then:
 \begin{enumerate}
 
\item $L\in \Max_s (E)$.

\item $\CA ( E, L)=\CA ( E, L, sim)=\{ E(L/M)\, | \, M\in \CF (L/K) \}$.

\item  The map 
$$\CF (L/K)\ra \CA ( E, L)=\CA (E, L, sim), \;\; M\mapsto C_E(M)=E(L/M)$$  is a bijection with inverse $B\mapsto C_E(B)$. 

\end{enumerate}
\end{theorem}  

\begin{proof} 1. Since $\Deg (E)=\Deg (M_n(K))=n=[L:K]$, we have that $L\in \Max_s (E)$. 

2. (i)  $\CA ( E, L)=\CA ( E, L, sim)$: Since $E\in \fC (K)$ and $L\in \Max_s (E)$, the statement (i) follows from Theorem \ref{A24Mar25}.(1).

(ii)  $\CA ( E, L, sim)=\{ C_E(M)=E(L/M)\, | \, M\in \CF (L/K) \}$: For all $M\in \CF (L/K)$, 
$C_E(M)=E(L/M)$. Now,  
$$
\CA ( E, L, sim)=\{ C_E(M)=E(L/M)\, | \, M\in \CF (L/K) \},
$$ 
by Proposition \ref{A24Mar25}.

3. Statement 3 follows from statement 2 and Proposition \ref{A24Mar25}.  
\end{proof}

\begin{corollary}\label{aB24Mar25}
Let  
 $L/K$ be a finite field extension,  $B\in \CA ( E, L)$ and $M:=C_E(B)\in \CF (L/K)$ (Theorem  \ref{B24Mar25}.(3)). Then:
 \begin{enumerate}
 
 \item  $B=E(L/M)\simeq M_{[L:M]}(M)\in \fC (M)$.
 
\item The field $L$ is the only (up to isomorphism)  simple $B$-module.

\item $\End_B(L)=M$.

\end{enumerate}
\end{corollary}

\begin{proof} 1.  By Theorem  \ref{B24Mar25}.(3), $B=E(L/M)\simeq M_{l:M]}(M)\in \fC (M)$.

2 and 3. Statements 2  and 3 follow from statement 1. 
\end{proof}

{\bf The subalgebra $L\rtimes G(L/K)$ of the endomorphism algebra $E$.}  Let $L/K$ be a finite field extension. Then a subalgebra of $E$, which is generated by the field $L$ and the automorphism group $G$,  is equal to the sum $\sum_{g\in G}Lg$  since $gl=g(l)g$ for all elements $l\in L$  and $g\in G$.

 \begin{theorem}\label{BC24Mar25}
Let $L/K$ be a finite field extension. Then:

\begin{enumerate}

\item  The field extension $L/L^{G}$ is a Galois field extension with Galois group $G(L/L^{G})=G$.

\item $\sum_{g\in G}Lg=\bigoplus_{g\in G}Lg=L\rtimes G=E(L/L^{G})\in \fC (L^{G})$, 
$$
\Deg (L\rtimes G)=[L:L^{G}]=|G(L/L^{G})|=|G|
$$ 
and 
$\dim_K(L\rtimes G)=[L:L^{G}]^2[L^{G}:K]=[L:K][L:L^{G}]$.

\item $L\in \Max_s(L\rtimes G)=\Max_s(E(L/L^{G}))$.

\item $L$ is the only (up to isomorphism) simple 
 $L\rtimes G$-module. 

\item $\End_{L\rtimes G}(L)=L^{G}$. 

\item The algebra $L\rtimes G=\bigoplus_{g\in G}Lg$ is a direct sum of  non-isomorphic simple $L$-bimodules. 

\end{enumerate}
\end{theorem}

\begin{proof} 1. By the Galois Theory, the field extension $L/L^G$ is a Galois field extension with Galois group $G(L/L^G)=G$.

2. Let $A:= \sum_{g\in G}Lg=\sum_{g\in G(L/L^G)}Lg$
 and $G=G(L/K)=G(L/L^G)$.
 
 (i) $A=\bigoplus_{g\in G}Lg=L\rtimes G$: 
 The automorphisms $g\in G$ are distinct (by the definition). Hence, they are $L$-linearly independent 
 as elements of the left $L$-vector space $E$, and so  $A=\bigoplus_{g\in G}Lg=L\rtimes G$.

(ii)  {\em The algebra $A$ is a simple algebra that contains the field $L$}: Clearly, $L\subseteq A$. So, it remains to prove simplicity of the algebra $A$. 
 For an element $a=\sum_{g\in G}l_g g\in A$, where $l_g\in L$, the set 
$$
\supp (a):=\{ g\in G\, | \, l_g\neq 0\}
$$
is called the {\em support} of the element $a$ and $|\supp (a)|$ is the number of elements in $\supp (a)$. Let $I$ be a nonzero ideal of the algebra $A$ and
$$ m:=\min \{ |\supp (a)| \, | \, 0\neq a\in I\}.$$It suffices to show that $m=1$ since then there is a nonzero element $l_gg\in I$, and so $g=l_g^{-1}l_gg\in I$ and $I=A$ since the element $g$ is  a unit of the algebra $A$ and its inverse $g^{-1}=g^{n-1}$ belongs to the ideal $I$ where $n$ is the order of the automorphism $g\in G$ (since $G$ is  a finite group, the order of  the automorphism $g$ is also finite). 

Suppose that $m\geq 2$ and $a=l_1g_1+\cdots +l_mg_m\in I$ where $l_1, \ldots , l_m\in L\backslash \{ 0\}$ and $g_1, \ldots , g_m\in G$ are distinct automorphisms. We seek a contradiction.  Then
$
ag_1^{-1}=l_1+l_2g_2g_1^{-1}+\cdots + l_mg_mg_1^{-1}\in I$. So, we may assume that $g_1=e$, i.e.
$$a=l_1+l_2g_2+\cdots +l_mg_m\in I.$$
Since $g_2\neq e$, we can find  an element, say $l\in L$, such that $g_2(l)\neq l$. Then,
$$
b:=[a,l]=l_2(g_2(l)-l)g_2+\cdots + l_m(g_m(l)-l)g_m\in I\backslash \{ 0\}\;\; {\rm and}\;\; |\supp (b)|\leq m-1<m,
$$ a contradiction.

(iii) $A=E(L/L^G)\in \fC (L^G)$: Notice that $A\subseteq E(l/L^G)$ and 
$$
\dim_{L^G}(A)=[L:L^G]^2=\dim_{L^G}(E(L/L^G)).
$$ Therefore, $A=E(L/L^G)\in \fC (L^G)$, and so 
$$
\Deg (A)=[L:L^G]=|G(L/L^G)|=|G|
$$ 
and 
$\dim_K(A)=[L:L^G]^2[L^G:K]=[L:K][L:L^G]$.

3.  Statement 3 follows from statement 1 (since $\Deg (A)=[L:L^G]$). 

4. The field $L$ is a simple $A=E(L/L^G)$-module (since $L\subseteq A$). Therefore, the field $L$ is the only (up to isomorphism)  simple $A=E(L/L^G)$-module.

5.  Statement 5 follows from the equality $A=E(L/L^G)$.

6. For all elements $l,l'\in L$ and $g\in G$, 
$$ lgl'=lg(l')g.$$ Clearly, for each element $g\in G$, the $L$-bimodule $Lg$ is simple. Suppose that $f: Lg\ra Lg'$ is a nonzero $L$-bimodule homomorphism. Then $f(g)=\l g'$ for some nonzero element $\l \in L$, and so 
$$g(l)\l g'=f(g(l)g)=f(gl)=f(g)l=\l g' l=\l g'(l)g'.$$
Hence, $g(l)=g'(l)$ for all elements $l\in L$. This means that $g=g'$  and statement 6 follows.
\end{proof}

 Each unit $u$  of an algebra $A$, determines the {\bf inner automorphism} $\o_u$ of the algebra $A$ which is given  by the rule: For all elements $a\in A$, $\o_u (a):=uau^{-1}$. Elements $g$ of the group $G$ are units of the algebra $L\rtimes G$ and 
$\o_g(lg')=g(l)gg'g^{-1}$ for all elements $g,g'\in G$ and $l\in L$. Therefore the map
$$
G\ra \Aut_K(L\rtimes G), \;\; g\mapsto \o_g
$$
is a group monomorphism. 

Let $ \CA (L\rtimes G, L)$ be the set of subalgebras of the algebra $L\rtimes G$ that contain the field $L$. A subalgebra $A$ of $L\rtimes G$ is called  $G{\bf-stable}$ if 
$$
\o_g (A)=A\;\; {\rm  for\; all\; elements}\;\;g\in G.
$$ 
Let $ \CA (L\rtimes G, L, G)$ be the set of $G$-stable subalgebras of the algebra $L\rtimes G$ that contain the field $L$. Corollary \ref{a24Mar25}  describes the sets  $\CA (L\rtimes G, L)$  and $\CA (L\rtimes G, L,  G)$.

 \begin{corollary}\label{a24Mar25}
Let $L/K$ be a finite field extension. Then:
\begin{enumerate}

\item  $ \CA (L\rtimes G, L)=\{L\rtimes H\, | \, H$ is a  subgroup of $G \}$.

\item  $ \CA (L\rtimes G,L,  G)=\{L\rtimes N\, | \, N$ is a normal subgroup of $G \}$.

\end{enumerate}
\end{corollary}

\begin{proof} 1. Suppose that  $A\in  \CA (L\rtimes G, L)$. Since $L\subseteq A$, the algebra $A$ is an $L$-bimodule. By Theorem \ref{BC24Mar25}.(6), the algebra $L\rtimes G=\bigoplus_{g\in G}Lg$ is a direct sum of  non-isomorphic simple $L$-bimodules. Therefore, 
$$
A=\bigoplus_{g\in H}Lg
$$ 
is a direct sum of  non-isomorphic simple $L$-bimodules for some subset $H$ of $G$ that contains the identity element $e$ of the group $G$. Since $A$ is an algebra and $LgLh=Lgh$ for all elements $g,h\in G$, the set $H$ is a submonoid of the {\em finite} group $G$. Therefore, the set $H$ is a subgroup of $G$, and statement 1 follows.

2. Suppose that  $A\in  \CA (L\rtimes G, L, G)$. By statement 1, $A=\bigoplus_{g\in N}Lg
$ for some subgroup $N$ of the group $G$. Since 
$gLng^{-1}=Lgng^{-1}$ for all elements $n\in N$ and $g\in G$, the  algebra $A$ belongs to the set $\CA (L\rtimes G, L, G)$ iff $gng^{-1}\in N$ for all elements $n\in N$ and $g\in G$ iff $N$ is a normal subgroup of $G$. 
\end{proof}


  \section{The Galois Theory from ring theoretic pooint of view} \label{AN-GT-NORMAL} 

The aim of the section is to prove  Theorem \ref{1Jun25} and  Theorem \ref{11Jun25}.  
 For a G-extension $L/K$, Theorem \ref{1Jun25}.(1) describes the set $\CA (E, L)$, Theorem \ref{1Jun25}.(2) gives the  Galois-type  Correspondence between $K$-subfields of $L$ and the algebras in $\CA (E, L)$ and as the result 
 Theorem \ref{1Jun25}.(3) gives the  Galois Correspondence for $K$-subfields of $L$. 
  Similarly, for   a G-extension $L/K$,  Theorem \ref{11Jun25}.(1) describes the set $\CA (E, L, G)$, Theorem \ref{11Jun25}.(2) gives the  Galois-type  Correspondence between Galois $K$-subfields of $L$ and the algebras in $\CA (E, L, G)$ and as the result  Theorem \ref{11Jun25}.(3) gives the  Galois Correspondence for Galois $K$-subfields of $L$.\\

 Proposition \ref{A8May25} shows that the classes of B-extensions and Galois field extensions coincide.
 
\begin{proposition}\label{A8May25}
A finite field extension $L/K$ is  a Galois field extension  iff $|G|=[L:K]$ iff  $L/K$ is a G-extension.
\end{proposition}

\begin{proof}
A finite field extension $L/K$ is a Galois field extension iff $|G|=[L:K]$ iff 
$$
\dim_K(L\rtimes G)=[L:K]^2\;\; (=\dim_K(E)
$$ 
iff 
 $L\rtimes G=E$ since $L\rtimes G\subseteq E$  iff  the field extension $L/K$ is a G-extension.
\end{proof}

 {\bf The Galois Correspondence for subfields.}  
\begin{definition} 
For a finite field extension $L/K$ and its intermediate subfield $K\subseteq M\subseteq L$, let
$$
G^M:= \{ g\in G\, | \, g(m)=m\;\; {\rm for\; all}\;\; m\in M \}=C_{G}(M).
$$
\end{definition} 
 By the definitions, and $G^M$ is a subgroup of the group $G$.

\begin{definition}
For a finite field extension $L/K$ and an algebra $A\in  \CA (E,L)$,  let
$$
 L^{A\cap G}:= \{  l \in L\, | \, g(l)=l\;\; {\rm for\; all}\;\; g\in A\cap G\}=C_{L}( A\cap G (L/K)).
 $$
\end{definition} 
 By the definition, the set $L^{A\cap G}$ is a subfield of $L$ that contain the field $K$ and the intersection $A\cap G$ is a subgroup of $G$.

\begin{lemma}\label{bc1Jun25}
Let $L/K$ be a finite field extension and $M\in \CF (L/K)$. Then $G(L/M)=G^M$. 
\end{lemma}

\begin{proof}  The equality  follows from the normality of the field extensions $L/K$ and $L/M$ and the inclusion $G(L/M)\subseteq G$. 
\end{proof}

For a G-extension $L/M$, Lemma \ref{ac1Jun25} shows that all the fields $L/M$, where $M\in \CF (L/K)$, are  G-extensions (using only the definition of a G-extension).

\begin{lemma}\label{ac1Jun25}
Let $L/K$ be a G-extension. Then for all fields $M\in \CF (L/K)$, $L/M$ is a G-extension. 
\end{lemma}

\begin{proof} Let  $M\in \CF (L/K)$  and $C:= C_{E}(M)$. By Theorem \ref{1Jun25}(1), $C=L\rtimes G(L/M)$ and 
$$
\dim_K(C)=\dim_K(L\rtimes G(L/M))=[L:K]|G(L/M)|.
$$
 Since each field $M\in \CF (L/K)$ is a simple $K$-subalgebra of $E$,   we have that $\dim_K(C)=\frac{\dim_K(E)}{[M:K]}$ (Theorem \ref{DCThm}.(3)) and so 
\begin{eqnarray*}
|G(L/M)| &=& \frac{[L:K]|G(L/M)|}{[L:K]}=\frac{\dim_K(C)}{[L:K]}=\frac{\dim_K(E)}{[M:K][L:K]}\\
 &=&\frac{[L:K]^2}{[M:K][L:K]}=[L:M].
\end{eqnarray*}
Therefore, 
$ \dim_M(L\rtimes G(L/M) )=[L:M]|G(L/M)|=[L:M]^2=\dim_M(E/M).$
This means that the field extension $L/M$ is a G-extension.
\end{proof}

For a finite  field extension $L/K$, let  $\CG (G)$ be the set of subgroups of the group  $G$. Theorem  \ref{1Jun25}.(3) is one of the main results of the classical Galois Theory.

 \begin{theorem}\label{1Jun25}
Let $L/K$ be a G-extension (i.e. a Galois extension, by Proposition \ref{A8May25}). Then:

\begin{enumerate}

\item $ \CA (E,L)=\{ C_{E}(M)=L\rtimes G(L/M)=L\rtimes G^M\, | \, L\in \CF (L/K)\}$. 

\item The map
$$
\CF (L/K)\ra \CA (E,L), \;\; M\mapsto  C_{E}(M)=L\rtimes G(L/M)=L\rtimes G^M
$$ 
is a bijection with inverse
$$
A\mapsto C_{E}(A)=L^{A\cap G}.
$$
\item {\bf (The  Galois Correspondence)} 
The map
$$
\CF (L/K)\ra \CG (G), \;\; M\mapsto  G(L/M)=G^M
$$ 
is a bijection with inverse $
H\mapsto L^H.$

\end{enumerate}

\end{theorem}

\begin{proof}
  
 1.  By  Theorem \ref{B24Mar25}.(3),
 $ \CA (E,L)=\{ C_{E}(M)=E(L/M)\, | \, L\in \CF (L/K)\}$.  The finite field extension $L/K$ is a B-extension, hence so is the field extension  $L/M$ for each field $M\in \CF (L/K)$. Therefore, 
$$
E(L/M)=L\rtimes G(L/M).
$$
By Lemma \ref{bc1Jun25}, $L\rtimes G(L/M)=L\rtimes G^M$, and statement 1 follows.

2.  In view of Theorem \ref{B24Mar25}.(3) and statement 1, it suffices that show that  the formula  for the inverse map in the theorem holds.
 Let $A\in \CA (E,L)$. By statement 1,  $A= L\rtimes G(L/M)$ for a unique field  $M\in \CF (L/K)$. The finite field extension $L/K$ is  normal/a B-extension. Hence,  $E=L\rtimes G$. It follows from the inclusions $G(L/M)\subseteq G$  and 
$$
A= L\rtimes G(L/M)\subseteq E=L\rtimes G=\bigoplus_{g\in G}Lg
$$
that $G(L/M)=A\cap   G$. Now we finish the proof of statement 2: 
\begin{eqnarray*}
 C_{E}(A)&=& C_{E}(L)\cap C_{E}( G(L/M))= L\cap  C_{E}( G(L/M))\\
  &=&  L^{G(L/M)}=L^{A\cap G}.
\end{eqnarray*}
3.  Recall that  Galois finite field extensions are G-extensions and vice versa (Proposition \ref{A8May25}). By statement 1, for each field $M\in \CF (L/K)$,  the centralizer 
 $C_{E}(M)=L\rtimes G(L/M)=L\rtimes G^M$ is uniquely determined by the group $G(L/M)=G^M$, and so statement 3 follows from statement 2.
\end{proof}

For a G-extension $L/K$ and its subfield $\G\in \CF (L/K)$, Lemma \ref{a19Jun25} is a criterion for the subgroup  $G(L/\G )$ of $G$ being a normal subgroup. 

\begin{lemma}\label{a19Jun25}
Let $L/K$ be a G-extension and $\G\in \CF (L/K)$. Then the subgroup $G(L/\G )$ of $G$ is a normal subgroup iff $g(\G )\subseteq \G$ for all elements $g\in G$. 
\end{lemma}

\begin{proof} $(\Rightarrow )$ Suppose that the subgroup $G(L/\G )$ of $G$ is a normal subgroup and $\g \in \G$. We have to show that $g(\g )\in \G$ for all elements $g\in G$. Let $h\in G(L/\G)$. Then $h':=g^{-1}hg\in G(L/\G)$, by the normality of the subgroup $G(L/\G )$ of $G$. Now, for all elements $\g \in \G$, 
$$ \g = h'(\g) = g^{-1}hg(\g),$$ i.e. $g (\g) = hg(\g)$ for all elements $h\in G(L/\G )$. By Theorem \ref{1Jun25}.(3), $g(\g) \in L^{G(L/\G )}=\G$, as required.

$(\Leftarrow )$ Suppose that $g(\G )\subseteq \G$ for all elements $g\in G$. Let $h\in G(L/\G )$. We have to show that $ghg^{-1}\in G(L/\G )$ for all elements $g\in G$. 
 For all elements $\g \in \G$, 
 $$
 ghg^{-1}(\g)=gg^{-1}(\g)=\g,
 $$
  and so $ghg^{-1}\in G(L/\G )$.
\end{proof}

A {\bf splitting field} of a set $S\subseteq K[x]$ of polynomials is  a field  that is obtained from $K$ by adding  the roots of all the polynomials in $S$.  For the set $S$, the slitting field is unique up to a $K$-isomorphism.

 Let $L/K$ be a finite field extension. Then its automorphism group $G$ acts on the field $L$,
 $$
 G\times L\ra L, \;\; (g,l)\mapsto g(l).
$$
 For each element $l\in L$, the subset of $L$,   $Gl:=\{ g(l)\, | \, g\in G\}$, is called the $G$-{\em orbit} of the element $l$.
 
 Lemma \ref{b19Jun25} shows that all G-extensions are splitting fields. 

\begin{lemma}\label{b19Jun25}
Let a finite field extension  $L/K$ be a G-extension. Then:
\begin{enumerate}

\item For each element $l\in L$, the minimal polynomial of the element $l$ over the field $K$ is equal to $m_l(x)=\prod_{l'\in Gl}(x-l')\in K[x]$. 

\item The field $L$ is a splitting field.

\item Suppose that a field extension $\G\in \CF (L/K)$ is a G-extension. Then $g(\G)=\G$ for all elements $g\in G$.

\end{enumerate} 
\end{lemma}

\begin{proof} 1. Since $L/K$ is a G-extension, $L^G=K$ (Theorem \ref{1Jun25}.(3)). By the definition, $$m_l(x)\in K^G[x]= K[x].$$ The polynomial  $m_l(x)$ is monic (i.e. the leading coefficient of $m_l(x)$ is 1). It remains  to show that the polynomial  $m_l(x)$ is irreducible over $K$. Suppose that a nonscalar polynomial  $f(x)\in K[x]$ is a divisor of the polynomial $m_l(x)$ and $l'\in Gl$ is its root. Then for all $g\in G$, $$0=g(0)=g(f(l'))=f(g(l')),$$ and so the set $Gl'=Gl$ consists of roots of the polynomial $f(x)$. Therefore, $f(x)=m_l(x)$.

2. The finite field extension $L/K$ is the composite of its subfields $K(l_1), \ldots , K(l_m)$ for          some elements $l_1, \ldots , l_m\in L$. By statement 1, the field $L$ is the splitting field of the polynomial $\prod_{i=1}^m m_{l_i}(x)$.

3. By statement 2,  the G-extension $\G\in \CF (L/K)$, is a splitting field over $K$. Therefore,  $g(\G)=\G$ for all elements $g\in G$. 
\end{proof}

For a G-extension, Theorem \ref{19Jun25} gives a characterizations of its G-subextensions. 

\begin{theorem}\label{19Jun25}
Let a finite field extension  $L/K$ be a G-extension and $\G\in \CF (L/K)$. Then the following statements are equivalent:
\begin{enumerate}

\item $\G/K$ is a G-extension.

\item   For all elements $g\in G$, $g(\G)=\G$. 

\item The group $G(L/\G)$ is a normal subgroup of $G$.

\end{enumerate} 
Suppose that $\G \in \CF (L/K)$ is a G-extension. Then (by statement 2) there is a short exact sequcne of group homomorphisms<
$$
1\ra G(L/\G)\ra G \stackrel{res}{\ra}G(\G / K)\ra 1
$$ where $res$ is the restriction map. In particular, $G(\G / K)\simeq G/G(L/\G)$.

\end{theorem}

\begin{proof} $(1\Rightarrow 2)$  Lemma \ref{b19Jun25}.(3).

 $(2\Leftrightarrow 3)$ Lemma \ref{a19Jun25}.

$(2\Rightarrow 1)$ By statement 2, there is a short exact sequence  of group homomorphisms:
$$
1\ra G(L/\G)\ra G \stackrel{res}{\ra}G(\G / K).
$$ 
We have to show that the map $res$ is surjective.
Since the field extensions $L/K$, $\G/K$ and $G/\G$ are G-extensions (Lemma \ref{ac1Jun25}), we have the equalities  
$$
|G|=[L:K],\;\; |G(\G / K)|=[\G :K]\;\; {\rm  and}\;\;|G(L/\G)|=[L:\G],
$$
 by Proposition \ref{A8May25}.  Clearly,  $H:=\im (res)\simeq G/G(L/\G )$ and 
$$
|H|=\frac{|G|}{|G(L/\G)|}=\frac{[L:K]}{[L:\G]}=[\G:K] =|G(M/K)|.
$$ 
Therefore,  $H=G(M/K)$ (since $H\subseteq G(\G /K)$ and $|H|=|G(M/K)|$), as required. 
\end{proof}

{\bf Galois subfields of   Galois finite field extensions.}  Theorem  \ref{11Jun25} gives  an order reversing bijection between 
 Galois subfields of a   Galois  finite field extension $L/K$ and the set $\CA (E,L, G)$ of $G$-stable subalgebras of the algebra $E$ that contains the field $L$. Recall that (Corollary \ref{a24Mar25}.(2)), 
 $$
 \CA (E,L, G)=\{L\rtimes  N \, | \, N\;\;{\rm is\; a\; normal\; subgroup\; of}\;\; G \}.
 $$
For a finite field extension $L/K$, let $\CG (L/K)$ be the set of  Galois  subfields $N/K$ of $L/K$ and 
 $\CN(G)$ be the set of normal subgroups of the group $ G$. Recall that the set $\CG (L/K)$ is the set of all $K$-subfields of $L$ that are  G-extensions (Proposition \ref{A8May25}).

\begin{theorem}\label{11Jun25}
Let $L/K$ be a  G-extension (i.e. a Galois extension). Then: 
\begin{enumerate}

\item $\CA (E,L, G)=\{C_E(\G )=L\rtimes G(L/\G)\, | \,$ $ \G\in \CG  (L/K)  \}=\{L\rtimes N\, | \, N\in \CN  (G)  \}$ and $\CN (G)= \{ G(L/\G)\, | \,$ $ \G\in \CG  (L/K)  \}$.

\item The map
$$
\CG (L/K)\ra  \CA (E,L, G), \;\; \G \mapsto  C_{E}(M)=L\rtimes G(L/\G) 
$$
 is a bijection with inverse
$$ A\mapsto C_{E}(A)=L^{A\cap G}.
$$
\item {\bf (The  Galois Correspondence for Galois subfields)} 
The map
$$
\CG (L/K)\ra \CN (G), \;\; \G\mapsto  G(L/\G)=G^{\G}
$$ 
is a bijection with inverse $
\G\mapsto L^{\G}.$
\end{enumerate}
\end{theorem}

\begin{proof}  By Proposition \ref{A8May25} and Theorem \ref{19Jun25}, 
$$
\CN (G)\stackrel{{\rm Thm.}\, \ref{19Jun25}}{=}
\{ G(L/\G)\, | \,\G\in \CF  (L/K),\; \G\;\; {\rm is\; a\;  G-extension}  \}
\stackrel{{\rm Pr.}\, \ref{A8May25}}{=} \{ G(L/\G)\, | \,\G\in \CG  (L/K)  \}.$$
Therefore,
$$\CA (E,L, G)\stackrel{{\rm Cor.}\, \ref{a24Mar25}.(2)}{=}\{L\rtimes N\, | \, N\in \CN  (G)  \}=\{C_E(\G )\stackrel{{\rm Thm.}\, \ref{1Jun25}.(1)}{=}L\rtimes G(L/\G)\, | \, \G\in \CG  (L/K)  \}.
$$
Now, statements 1--3 of the theorem follows from statements 1--3 of Theorem \ref{1Jun25}, respectively.
\end{proof}

 {\bf Licence.} For the purpose of open access, the author has applied a Creative Commons Attribution (CC BY) licence to any Author Accepted Manuscript version arising from this submission.

{\bf Declaration of interests.} The authors declare that they have no known competing financial interests or personal relationships that could have appeared to influence the work reported in this paper.


{\bf Data availability statement.} Data sharing not applicable – no new data generated.

{\bf Funding.} This research received no specific grant from any funding agency in the public, commercial, or not-for-profit sectors.

\small{

School of Mathematical  and Physical Sciences

Division of Mathematics

University of Sheffield

Hicks Building

Sheffield S3 7RH

UK

email: v.bavula@sheffield.ac.uk}

\end{document}